\documentclass[10pt]{amsart}

\usepackage{amssymb}
\usepackage{enumerate}
\usepackage[dvips]{graphicx}
\DeclareGraphicsExtensions{.eps,.art,.ART,.ps}

\usepackage{amsmath, amsthm, amsfonts, color}
\usepackage{srcltx}

\newcommand{\ds}{\displaystyle}

\newcommand{\eqskip}{ \vspace*{2mm}\\ }

\newcommand{\bt}{\beta}

\newcommand{\xx}{\bar{x}}

\newtheorem{theorem}{Theorem}
\newtheorem{thm}{Theorem}[section]

\newtheorem{lemma}[thm]{Lemma}

\newtheorem{cor}[thm]{Corollary}
\theoremstyle{definition}

\theoremstyle{remark}
\newtheorem{examp}{Example}

\def\tht{\theta}
\def\Om{\Omega}

\def\e{\varepsilon}

\def\G{\Gamma}
\def\l{\lambda}
\def\p{\partial}
\def\D{\Delta}

\def\E{\mbox{\rm e}}
\def\a{\alpha}
\def\b{\beta}
\def\Ups{\Upsilon}

\def\d{\delta}
\def\L{\Lambda}
\def\z{\zeta}

\def\Odr{\mathcal{O}}
\def\H{W_2}

\def\di{\,\mathrm{d}}

\numberwithin{equation}{section}

\title[Asymptotics for Dirichlet eigenvalues and eigenfunctions]
{Singular asymptotic expansions for Dirichlet eigenvalues and
eigenfunctions of the Laplacian on thin planar domains}
\author{Denis Borisov \and Pedro Freitas}
\address{
Faculty of Mathematics, Chemnitz University of Technology,
D-09107, Chemnitz, Germany,  {\rm and} Department of Physics and
Mathematics, Bashkir State Pedagogical University, October rev.
st., 3a, 450000, Ufa, Russia }\email{borisovdi@yandex.ru}
\address{Department of Mathematics, Faculdade de Motricidade Humana (TU Lisbon) {\rm and}
Group of Mathematical Physics of the University of Lisbon\\
Complexo Interdisciplinar,
Av.~Prof.~Gama Pinto~2\\
P-1649-003 Lisboa,
Portugal}\email{freitas@cii.fc.ul.pt}
\date{\today}
\subjclass[2000]{Primary 35P15; Secondary 35J05}

\thanks{D.B. was partially supported by RFBR (06-01-00138)
and gratefully
acknowledges the support from Deligne 2004 Balzan prize in
mathematics and the grant of Republic Bashkortostan for young
scientists and young scientific collectives. P.F. was partially
supported by FCT/POCTI/FEDER. Part of this work was done during
the visit of D.B. to Universidade de Lisboa; he is grateful for
the hospitality extended to him.
 }

\begin{document}

\allowdisplaybreaks

\begin{abstract}
We consider the Laplace operator with Dirichlet boundary
conditions on a planar domain and study the effect that
performing a scaling in one direction has on the spectrum. We
derive the asymptotic expansion for the eigenvalues and
corresponding eigenfunctions as a function of the scaling
parameter around zero. This method allows us, for instance, to
obtain an approximation for the first Dirichlet eigenvalue for a
large class of planar domains, under very mild assumptions.
\end{abstract}
%
%

\maketitle

\section{Introduction}

The study of the spectrum of the Laplace operator on thin
domains has received much attention in the mathematical
literature over the last few years. Apart from the connection to
certain physical systems such as quantum waveguides, this limit
situation is also of interest as it may provide insight into
certain questions related to the spectrum of the Laplacian --
see the recent paper by Friedlander and Solomyak~\cite{frso} for
some references on both counts.

The purpose of the present paper is to study this spectrum in
the singular limit around a line segment in the plane. More
precisely, given a planar domain we consider its scaling in one
direction, so that in the limit we have a line segment
orthogonal to this direction. In particular, and under mild
smooth assumptions on the domain, we derive the series expansion
for the first Dirichlet eigenvalue thus showing that the
coefficients in this series have a simple explicit expression in
terms of the functions defining the domain -- see
Theorem~\ref{th5.1} below.

Due to the notorious lack of explicit expressions for Dirichlet
eigenvalues of planar domains, this seems to be a possible path
towards obtaining information about such eigenvalues, and indeed
it was one of the motivations behind our work. As an example,
consider the case of the family of ellipses
centred at $(1/2,0)$ and with axes $1/2$ and $\e/2$. In this
case our results yield that
\begin{equation}\label{diskasympt}
\lambda_{1}(\e) = \frac{\ds \pi^{2}}{\ds
\e^2}+\frac{\ds 2\pi}{\ds \e} +3 + \left(\frac{\ds 11}{\ds
2\pi}+\frac{\ds \pi}{\ds 3}\right)\e + \Odr(\e^2), \;\; \mbox{
as } \e\to0,
\end{equation}
thus complementing the asymptotic expansion for the first
eigenvalue of ellipses around the disk derived by Joseph
in~\cite{jose}. Further terms in the expansion may be obtained
by means of Lemma~\ref{lm5.7}. The graphs of the four--term asymptotic
expansions for two other examples, namely the lemniscate and the
bean curve are shown in Figure~\ref{fig:lemnbean}. For a discussion of
these and other examples, see the last section.

\begin{figure}[!ht]\label{fig:lemnbean}
\centering
\includegraphics[width=0.47\textwidth]{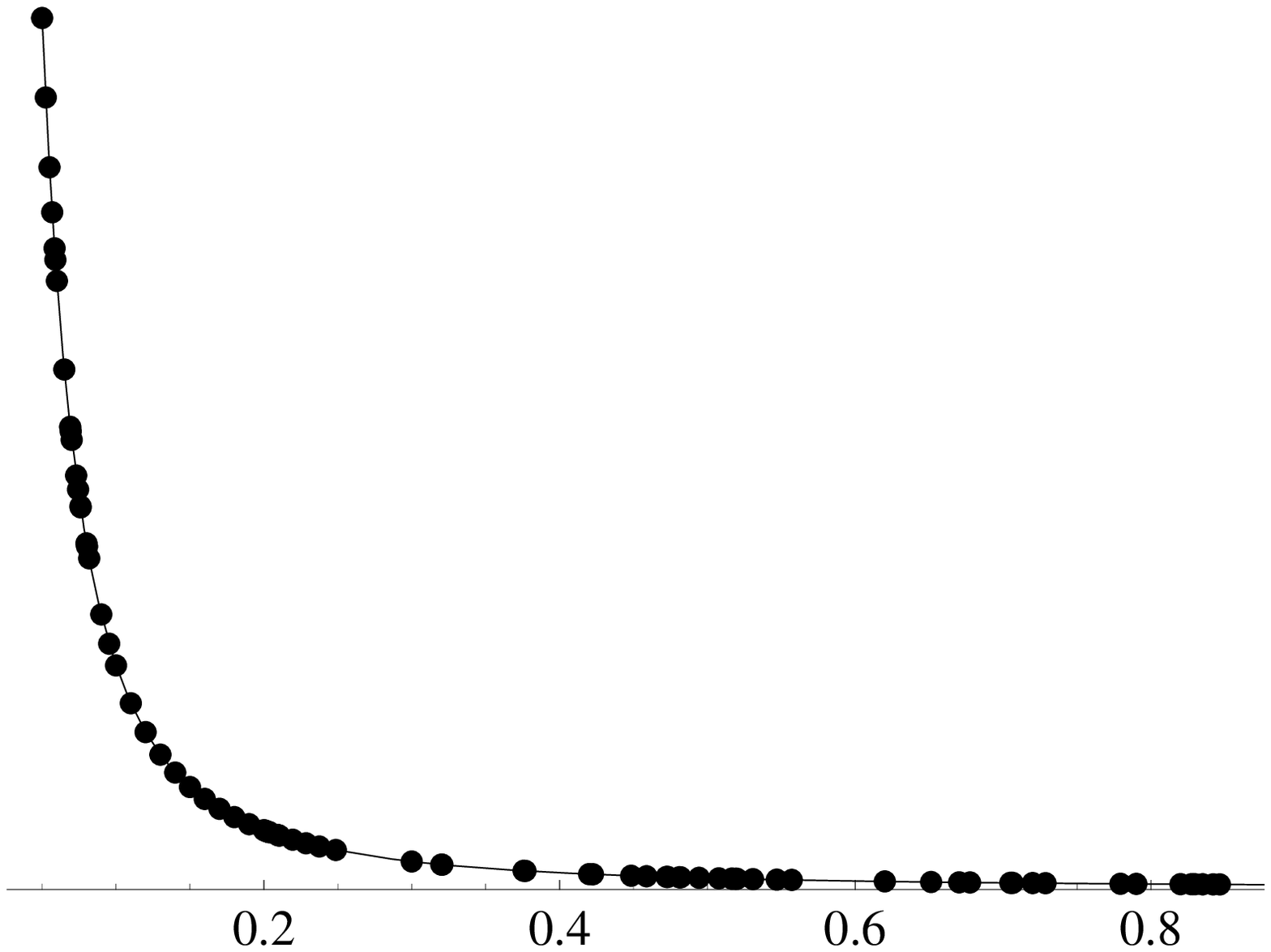}
\includegraphics[width=0.47\textwidth]{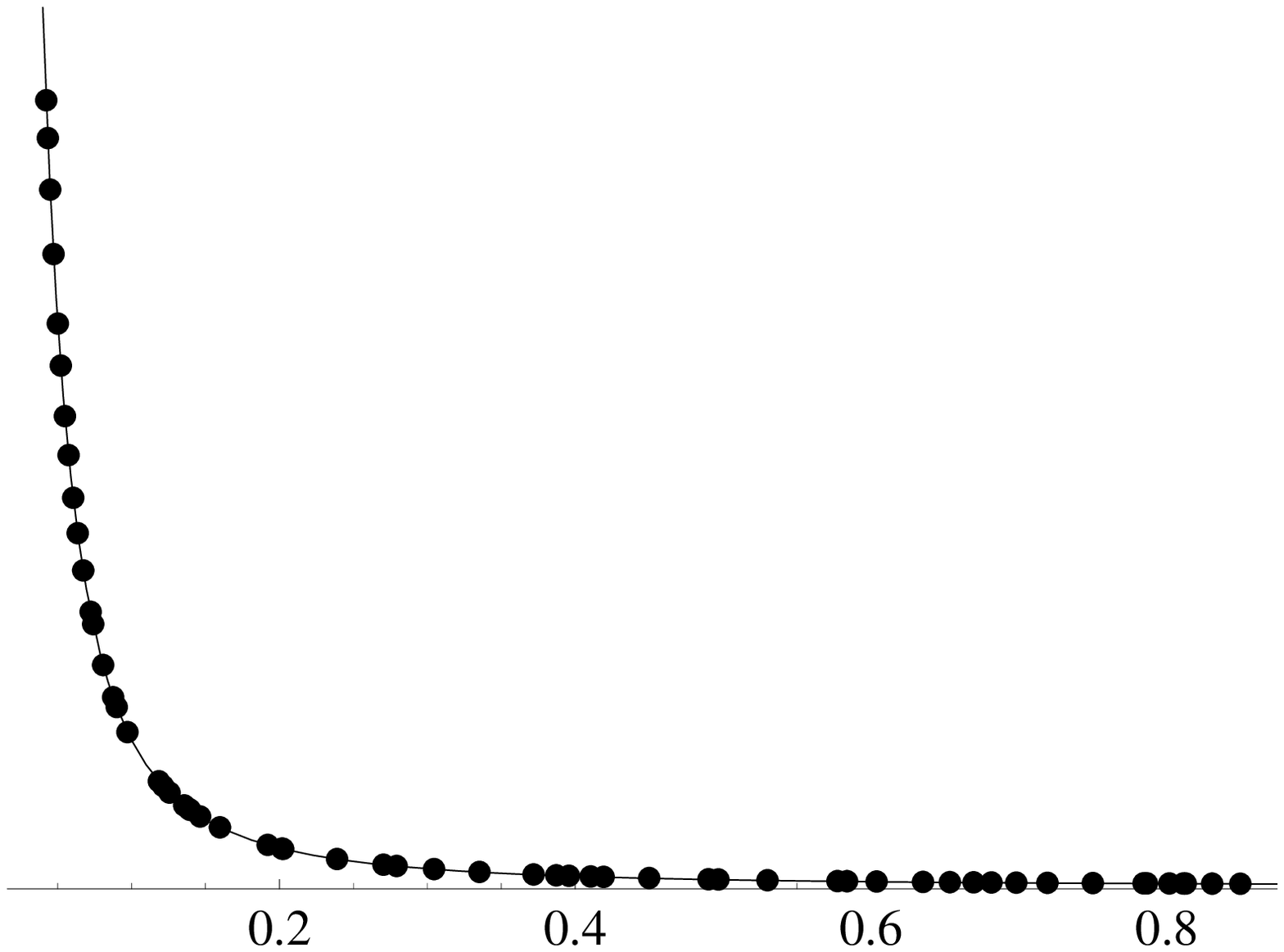}
\caption{Graphs of the four--term asymptotic expansions obtained for the first
eigenvalue of the lemniscate
and the bean curve in Examples~\ref{ex:lemn} and~\ref{ex:bean} presented in Section~\ref{examp}.
The points are numerical approximations to the eigenvalue with an error not greater than $10^{-5}$ .}
\end{figure}

The advantages of our method are twofold. On the one hand, it
does not require any {\it a priori} knowledge neither of the
eigenvalues nor of the eigenfunctions of a particular domain, as
the perturbation is always made around the singular case of the
line segment. In the above example, for instance, Joseph's
method requires knowledge of the first eigenvalue of the disk.
This limits applications of the same type of argument to general
examples. Furthermore, it is not evident that even in cases
where the eigenvalue of the unperturbed domain is known one can
determine explicitly the value of the coefficients of the
perturbation around this domain in terms of known constants.
This is illustrated by the example in Section~5 in~\cite{jose},
where one of the coefficients in the expansion gives rise to a
double series. On the other hand, and by its own nature, our
method is particularly suited to dealing with long and thin
domains where numerical methods will tend to have problems.

The fact that we consider a singular perturbation does pose,
however, some difficulties. In order to understand these, let's
begin by observing that the behaviour of the asymptotic series
obtained for the eigenvalues depends on what happens at the
point of the line segment where the vertical width is maximal,
say $\xx$. More precisely, the coefficients appearing in the
expansion will depend on the value that the functions defining
the domain and their derivatives take at $\xx$, but not on
what happens away from this point. This is not surprising, since
it is in this region that the first eigenfunction will tend to
concentrate as the parameter goes to zero. In order to derive
the asymptotics, we assumed that there existed only one such
point and that the domain was $\mathcal{C}^{\infty}$ smooth in a
small neighbourhood of $\xx$. While the former condition may
be somewhat relaxed, and a finite number of points of maximal
height could, in principle, be considered, if the latter
condition is not satisfied then the powers appearing in the
asymptotic series will actually be different.

A first point is thus that the type of expansion obtained will
depend on the local behaviour at $\xx$. To illustrate what may
happen, consider the following example studied by the second
author using a different approach~\cite{frei} . Let $T_{\bt}$ be
an obtuse isosceles triangle where the largest side is assumed
to have fixed unit length, and where the equal angle $\bt$
approaches zero. In this case we have
\[
\begin{array}{lll}
\lambda_{1}(T_{\bt}) & = & \frac{\ds 4\pi^{2}}{\ds \bt^{2}}-
\frac{\ds 4\cdot 2^{2/3}a_{1}^{\prime}\pi^{4/3}}{\ds \bt^{4/3}}
+\Odr(\bt^{-2/3}) \;\; \mbox{ as } \bt \to 0,
\end{array}
\]
where $a_{1}^{\prime}\approx -1.01879$ is the first negative
zero of the first derivative of the Airy function of the first
kind. We thus see that having the maximal width at a corner
point introduces fractional powers in the expansion.

On the other hand, it can be seen from our results that what
happens away from $\xx$ does not influence the coefficients
appearing in the expansion in Theorem~\ref{th5.1}. Thus in the
example given above for ellipses, changing the domain outside a
band containing the mid-region will yield precisely the same
coefficients although the expression for the first eigenvalue as
a function of $\e$ will certainly be different. This means that
the series described in Theorem~\ref{th5.1} cannot, in general,
converge to the desired eigenvalue, and that there should exist
a {\em tail} term which goes to zero faster than any power of
$\e$, which our analysis does not allow us to recover. An
interesting question is thus to determine the nature of such a
term and when can we ensure that the series expansion derived
here is actually convergent to the corresponding eigenvalue.

Note that although we address the same problem as
in~\cite{frso}, both our approach and results are quite
different from those in that paper. In particular, we obtain the
full asymptotic expansion for the eigenvalues and
eigenfunctions. Furthermore, the coefficients in these
expansions are obtained as a simple explicit function of the
value of the functions involved and their derivatives at the
point where the function $H$ takes its maximum. Another
difference is that we consider a two-parameter set of
eigenvalues, while the method used in~\cite{frso} leads to a
one-parameter set. On the other hand, our approach required the function $H$ to be
smooth in a neighbourhood of this point, while the results
in~\cite{frso} are more general in this respect, allowing for
the existence of corners, for instance. We hope to be able to
consider this situation in a forthcoming paper. Furthermore,
in~\cite{frso} the authors actually prove convergence in the
norm of the resolvents, which is a stronger property than
convergence of the eigenvalues and eigenfunctions.

The plan of the paper is as follows. In the next section we
state the results in the paper. In Section~\ref{proof1} we prove
the general form of the asymptotic expansions for eigenvalues
and eigenfunctions. The proof is split into two parts. First we
construct the asymptotics expansions formally. The main idea
here is to use the ansatz of boundary layer type which localizes
in a vicinity of the point $\xx$ mentioned above. After formal
constructing, the formal expansions are justified, i.e., the
estimates for the error term are obtained. Here the main tool is
Vishik-Lyusternik's lemma.
Section~\ref{proof2} contains the study of the expansion for the
first eigenvalue. This consists in identifying the eigenvalue
which corresponds to a positive eigenfunction, and then
obtaining the explicit form of the terms in the expansion of the
eigenvalue. In the last section we consider some applications of
our results to specific domains.

\section{Statement of results}

Let $h_\pm=h_\pm(x_1)\in \mathcal{C}[0,1]$ 
be arbitrary functions and write $H(x_1):={h}_+(x_1)+{h}_-(x_1)$
for $x_1$ in $[0,1]$. We shall consider the thin domain defined
by
\[
\Om_\e = \{x: x_1\in(0,1), -\e {h}_-(x_1)<x_2<\e {h}_+(x_1)\},
\]
for which we assume that the function $h$ attains its global
maximum at a single point $\xx\in(0,1)$, and that $H(x_1)>0$ for
$x_1\in(0,1)$. Note that the cases where either $H(0)=0$ or
$H(1)=0$ are not excluded. We also assume that the functions
$h_\pm$ are infinitely differentiable in a small neighbourhood
of $\xx$, say $(\xx-\d,\xx)$. For the sake of brevity, in
what follows we shall write $H_0:=H(\xx)$ and denote by $H_i$
the $i$-th derivative of $h$ at $\xx$. In the same way we
denote the derivatives of ${h}_-$ by $h_{i}$. We shall assume
that there exists $k\geqslant 1$ such that
\begin{equation}\label{5.0}
H_{i}=0,\quad i=1,\ldots,2k-1,\qquad H_{2k}<0.
\end{equation}

Our aim is to study the asymptotic behaviour of the eigenvalues
and the eigenfunctions of the Dirichlet Laplacian
$-\D^{D}_{\Om_\e}$ in $\Om_\e$. 
Let $\chi=\chi(t)\in \mathcal{C}^\infty(\mathbb{R})$ be a
non-negative cut-off function which equals one when
$|\xi_1-\xx|<\d/3$ and vanishes for $|t-\xx|>\d/2$. We denote
$\Om^\d_\e:=\Om_\e\cap\{x: |t-\xx|<\d\}$.

The main result of the paper is the following.
\begin{theorem}\label{th5.1}
Under the above conditions, 
there exist the eigenvalues $\l_{n,m}(\e)$, $n,m\in\mathbb{N}$,
of the operator $-\D^{D}_{\Om_\e}$, whose asymptotic expansions
as $\e$ goes to zero reads as
\begin{equation}\label{5.65}
\l_{n,m}(\e)=\e^{-2}c_0^{(n,m)}+\e^{-2}\sum\limits_{i=2k}^{\infty}
\eta^i c_i^{(n,m)},\quad \eta:=\e^{\a},\quad \a:=\frac{1}{k+1},
\end{equation}
where
\begin{equation}\label{5.66}
\begin{aligned}
&c_0^{(n,m)}=\frac{\pi^2 n^2}{H_0^2},\quad
c_{2k}^{(n,m)}=\L_{n,m},\quad c_{2k+1}^{(n,m)}=0,
\\
&c_{2k+2}^{(n,m)}=
\frac{\pi^2n^2h_{1}^2}{H_0^2}-\frac{2\pi^2n^2 H_{2k+1}}{(2k+1)!
H_0^3}(\xi_1^{2k+1}\Psi_1^{(n,m)},\Phi_{n,m})_{L_2(\mathbb{R})}
\\
&-\frac{2\pi^2n^2H_{2k+2}}{(2k+2)! H_0^3}
\|\xi_1^{k+1}\Phi_{n,m}\|_{L_2(\mathbb{R})}^2
+\boldsymbol{\d}_{1k} \frac{3\pi^2n^2H_{2k}^2}{\big((2k)!\big)^2
H_0^4} \|\xi_1^{2k}\Phi_{n,m}\|_{L_2(\mathbb{R})}^2.
\end{aligned}
\end{equation}
Here $\L_{n,m}$ and $\Phi_{n,m}$  are the eigenvalues and the
associated orthonormalized in $L_2(\mathbb{R})$ eigenfunctions
of the operator
\begin{equation*}
G_n:=-\frac{d^2}{d\xi_1^2}- \frac{2\pi^2 n^2
H_{2k}}{(2k)!H_0^3}\xi_1^{2k}
\end{equation*}
in $L_2(\mathbb{R})$, $\boldsymbol{\d}_{1k}$ is the Kronecker
delta,  and $\Psi_1^{(n,m)}\in C^\infty(\mathbb{R})\cap
L_2(\mathbb{R})$ is the exponentially decaying solution to
\begin{equation*}
(G_n-\L_{n,m})\Psi_1^{(n,m)}=
\frac{2\pi^2n^2H_{2k+1}\xi_1^{2k+1}}{(2k+1)!H_0^3} \Phi_{n,m},
\end{equation*}
which is orthogonal to $\Phi_{n,m}$ in $L_2(\mathbb{R})$.

 The asymptotics of the eigenfunctions
$\psi_{n,m}(x,\e)$ associated with $\l_{n,m}(\e)$ read as
follows
\begin{equation}
\psi_{n,m}(x,\e)=\chi(x_1) \sum\limits_{i=0}^{\infty}
\eta^i\psi_i^{(n,m)}(\xi),\label{5.3}
\end{equation}
in $\H^1(\Om_\e)$-norm and $\H^2(\Om_\e^{\d/3})$-norm, where
\begin{align*}
&\xi_1=\frac{x_1-\xx}{\e^\a},\quad \xi_2=\frac{x_2+\e
{h}_-(x_1)}{\e H(x_1)}, 
\\
&\psi_0^{(n,m)}(\xi)=\Phi_{n,m}(\xi_1)\sin\pi n\xi_2,\quad
\psi_1^{(n,m)}(\xi)=\Psi_1^{(n,m)}(\xi_1)\sin\pi n\xi_2.
\end{align*}
The remaining coefficients of the series (\ref{5.65}),
(\ref{5.3}) are determined by Lemma~\ref{lm5.7}.
\end{theorem}

The next result gives explicit expressions for the first four
nonvanishing terms in the asymptotic expansion of the first
eigenvalue in terms of the functions $h_\pm$ and their
derivatives at $\xx$ in the case where $H_{2}$ is negative.
\begin{theorem}\label{th5.2}
The eigenvalue $\l_{1,1}(\e)$ is the ground state. In
particular, if $k=1$, it will have the expansion
\[
\l_{1,1}(\e) = \frac{\ds c_{0}^{(1,1)}}{\ds \e^2}+\frac{\ds c_2^{(1,1)}}{\ds \e}
+ c_4^{(1,1)} + c_6^{(1,1)} \e + \Odr(\e^2), \;\; \mbox{ as } \e\to0,
\]
where
\begin{equation}
\begin{aligned}
& c_{0}^{(1,1)} = \frac{\ds \pi^{2}}{\ds H_{0}^{2}}, \quad
c_2^{(1,1)}=\frac{\pi\big(-H_2\big)^{1/2}}{H_0^{3/2}},
\\
&c_4^{(1,1)}=\frac{\pi^2h_{1}^2}{H_0^2}-\frac{9 H_2}{16
H_0}-\frac{11H_3^2}{144H_2^2}+ \frac{H_4 }{16 H_2},
\\
&c_6^{(1,1)}=\frac{H_0^{3/2}}{\pi\big(-H_2\big)^{1/2}} \Bigg(
\frac{\pi^2 h_{2}^2}{2H_0^2}
-\frac{\pi^2(h_{1}H_3+H_2^2)h_{2}}{2H_0^2
H_2}+\frac{83H_2^2}{256H_0^2}
\\
&\hphantom{c_6^{(1,1)}=\Bigg(} +\frac{19H_3^2H_4}{384H_2^3}
-\frac{155H_3^4}{6912H_2^4}+\frac{29 H_3^2}{384H_0 H_2}
-\frac{9H_4}{128 H_0}-\frac{13H_3H_5}{576H_2^2}
\\
&\hphantom{c_6^{(1,1)}=\Bigg(} -\frac{7H_4^2}{768H_2^2}
+\frac{H_6}{192H_2}+\frac{\pi^2H_2^2}{6H_0^2} - \frac{\pi^2 H_2
h_{1}^2 }{2H_0^3}+ \frac{\pi^2h_{3}h_{1}}{2H_0^2} \Bigg).
\end{aligned}\label{5.2}
\end{equation}
\end{theorem}

In the case of higher eigenvalues, which of the eigenvalues
$\lambda_{n,m}(\e)$ corresponds to a given eigenvalue
$\lambda_{\ell}$ will depend on the value of $\e$. Furthermore,
generally speaking, the eigenvalues given in Theorem~\ref{th5.1}
do not exhaust the whole spectrum. Indeed, assume that the
function $H$ has a local maximum at a point
$\widetilde{x}\not=\xx$. Then one can reproduce the proof of
Theorem~\ref{th5.1} and show that there exists an additional
infinite two-parametric set of the eigenvalues associated with
$\widetilde{x}$; the corresponding eigenfunctions are localized
in a vicinity of $\widetilde{x}$. On the other hand, we
conjecture that given $m$, there exists $\e_0$ small enough so
that for $\e<\e_0$ the first $m$ eigenvalues of $-\D_{\Om_\e}^D$
are the eigenvalues $\l_{1,m}$ described in Theorem~\ref{th5.1}.
This is true for $m=1$ due to Theorem~\ref{th5.2}, but we are
not able to prove this fact even for $m=2$. We should also note
that this fact was proven in \cite{frso} but under the
additional assumption that the width does not vanish at the
end-points and $h_-\equiv0$. In \cite{frso} the described fact
follows from the uniform resolvent convergence.

Although we can not prove that the second eigenvalue of
$-\D_{\Om_\e}^D$ is $\l_{1,2}$,
we must certainly have $\lambda_{2}\leq
\lambda_{1,2}$ which is sufficient to allow us to obtain the
following result for the spectral gap $\gamma :=
\lambda_{2}-\lambda_{1}$.
\begin{cor}
Under the above hypothesis and with $k=1$, the quantity
$\e\gamma(\e)$ remains bounded as $\e$ goes to zero. More precisely,
\[
\gamma(\e)\leqslant \frac{\ds 2\pi
(-H_{2})^{1/2}}{\ds H_{0}^{3/2}\e} + \Odr(1), \mbox{ as } \e\to
0.
\]
\end{cor}


We remark that, in general, the spectral gap is unbounded when
either the diameter or the area are kept constant, the simplest
example being that of a circular sector where the opening angle
approaches zero. From the results for obtuse isosceles triangles
in~\cite{frei} we know this to be the case also for
one--parameter families of domains of the type considered here.
What this result shows is that under the conditions above the
gap will remain bounded as $\e$ goes to zero if we keep the area
fixed. To see this, it is sufficient to note that $|\Om_{\e}| =
\e |\Om_{1}|$ and thus $|\Om_{\e}|\gamma(\e)$ is also bounded. In
particular, this shows that regularity at the point of maximum
height plays an important role in bounding the gap. Note that
this boundedness is not uniform on the domain. Also, if we fix
the diameter instead of the area then the gap will still go to
infinity as $\e$ goes to zero.

\section{Proof of Theorem~\ref{th5.1}\label{proof1}}

Let $\l_\e$ and $\psi_\e$ be an eigenvalue and an associated
eigenfunction of $-\D^{D}_{\Om_\e}$, respectively. We construct the
asymptotics for $\l_\e$ and $\psi_\e$ as the series
\begin{equation}\label{5.3a}
\l_\e=\e^{-2}\mu_\e,\quad \mu_\e=
c_0+\sum\limits_{i=2k}^{\infty}\eta^i c_i,\quad \psi_\e(\xi)=
\sum\limits_{i=0}^{\infty} \eta^i\psi_i(\xi),
\end{equation}
where $c_i$ and $\psi_i$ are the coefficients and functions to
be determined. 

In what follows we will show that the function $\psi_\e$ is
exponentially small (with respect to $\e$) outside $\Om^\d_\e$. This is
why we are interested only in determining $\psi_i$ on
$\Om^\d_\e$. After passing to the variables $\xi$, the domain
$\Om^\d_\e$ becomes $\{\xi: |\xi_1|<\d\eta^{-1}, 0<\xi_2<1\}$.
As $\eta\to0$, the latter domain ''tends'' to the strip
$\Pi:=\{\xi: 0<\xi_2<1\}$. This is why we will consider the
functions $\psi_i$ as defined on $\Pi$. The mentioned
exponential decaying of the eigenfunction is implied by the fact
that all the coefficients $\psi_i$ decay exponentially as
$|\xi_1|\to\pm\infty$, in other words, we postulate the latter 
for $\psi_i$.

Having the made assumptions in mind, we rewrite the eigenvalue
equation for $\psi_\e$ considered in $\Om^\d_\e$ in the
variables $\xi$,
\begin{equation}
\begin{gathered}
\left(-K_{11}
\frac{\p^2}{\p\xi_1^2}-2K_{12}\frac{\p^2}{\p\xi_1\p\xi_2}
-K_{22}\frac{\p^2}{\p\xi_2^2}-K_{2}\frac{\p}{\p\xi_2}
\right)\psi_\e=\mu_\e\psi_\e\quad\text{in}\quad\Pi,
\\
\psi_\e=0\quad\text{on}\quad\p\Pi,
\end{gathered}
\label{5.53}
\end{equation}
where $K_{ij}=K_{ij}(\xi_1\eta,\xi_2,\eta)$,
$K_i=K_i(\xi_1\eta,\xi_2,\eta)$, and
\begin{equation}\label{5.53a}
\begin{aligned}
&K_{11}(z,\eta)=\eta^{2k},\quad K_{12}(z,\eta)=\eta^{2k+1}\frac{
h'_-(\xx+z_1)- z_2 H'(\xx+z_1)}{H(\xx+z_1)},
\\
&K_{22}(z,\eta)=\frac{ 1+\eta^{2k+2} \left(h'_-(\xx+z_1)-z_2
H'(\xx+z_1)\right)^2}{h^2(\xx+z_1)},
\\
&K_{2}(z,\eta)=\frac{\eta^{2k+2}}{h^2(\xx+z_1)} \Big(
h''_-(\xx+z_1)H(\xx+z_1)-2{h}_-'(\xx+z_1) H'(\xx+z_1)
\\
&\hphantom{ K_{2}(z,\eta)=\frac{\eta^{2k+2}}{h^2(\xx+z_1)} \Big(
} +2z_2\big(H'(\xx+z_1)\big)^2-z_2 H''(\xx+z_1) H(\xx+z_1)
\Big).
\end{aligned}
\end{equation}
Now we expand the coefficients $K_{12}$, $K_{22}$, and $K_2$
into the Taylor series with respect to $\eta$, taking into account
(\ref{5.0}),
\begin{equation}\label{5.51}
\begin{aligned}
&K_{12}(\xi_1\eta,\xi_2,\eta)=\sum\limits_{i=2k+1}^{\infty}\eta^i
\left(P_{12}^{(i)}(\xi_1)+\xi_2 Q_{12}^{(i)}(\xi_1)\right)
\\
&K_{22}(\xi_1\eta,\xi_2,\eta)=H_0^{-2}+
\sum\limits_{i=2k}^{\infty}\eta^i
\left(P_{22}^{(i)}(\xi_1)+\xi_2 Q_{22}^{(i)}(\xi_1)+\xi_2^2
R_{22}^{(i)}(\xi_1)\right),
\\
&K_{2}(\xi_1\eta,\xi_2,\eta)=\sum\limits_{i=2k+2}^{\infty}\eta^i
\left(P_{2}^{(i)}(\xi_1)+\xi_2 Q_{2}^{(i)}(\xi_1)\right),
\end{aligned}
\end{equation}
where $P^{i}_{12}$, $P^{i}_{22}$,  $P^{i}_{2}$, $Q^{i}_{12}$,
$Q^{i}_{22}$,  $Q^{i}_{2}$ are certain polynomials with respect to
$\xi_1$, and, in particular,
\begin{equation}\label{5.52}
\begin{aligned}
& P_{12}^{(2k+1)}=\frac{h_{1}}{H_0},\quad
P_{12}^{(2k+2)}=\frac{h_{2}}{H_0}\xi_1,
\\
&Q_{12}^{(2k+1)}=0, \quad Q_{12}^{(2k+2)}=-
\frac{H_2}{H_0}\xi_1,
\\
& P_{22}^{(2k)}=-\frac{2H_{2k}\xi_1^{2k}}{(2k)!H_0^3}, \quad
P_{22}^{(2k+1)}=-\frac{2H_{2k+1}\xi_1^{2k+1}}{(2k+1)!H_0^3},
\\
&P_{22}^{(2k+2)}=-\frac{2H_{2k+2}\xi_1^{2k+2}}{(2k+2)!H_0^3}
+\boldsymbol{\d}_{1k}
\frac{3H_{2k}^2\xi_1^{4k}}{\big((2k)!\big)^2 H_0^4
}+\frac{h_{1}^2}{H_0^2}
\\
&Q_{22}^{(2k)}=Q_{22}^{(2k+1)}=Q_{22}^{(2k+2)}=0, \quad
R_{22}^{(2k)}=R_{22}^{(2k+1)}=R_{22}^{(2k+2)}=0,
\\
&P_{2}^{(2k+2)}=\frac{h_{2}}{H_0},\quad
Q_{2}^{(2k+2)}=-\frac{H_2}{H_0}.
\end{aligned}
\end{equation}

We substitute (\ref{5.3a}) and (\ref{5.51}) into (\ref{5.53})
and evaluate the coefficients of the same powers of $\eta$
taking into account (\ref{5.52}).  It leads us to the system of
the boundary value problems,
\begin{align}
&\frac{1}{H_0^2}\frac{\p^2\psi_i}{\p\xi_2^2}+c_0\psi_i=0\quad
\text{in}\quad\Pi,\qquad \psi_i=0\quad\text{on}\quad\p\Pi,\quad
i=0,\ldots,2k-1,\label{5.54}
\\
&
\begin{aligned}
-\frac{1}{H_0^2}\frac{\p^2\psi_{2k}}{\p\xi_2^2}-\frac{\pi^2
n^2}{H_0^2}\psi_{2k}&=\frac{\p^2\psi_0}{\p\xi_1^2}-
\frac{2H_{2k}\xi_1^{2k}}{(2k)!H_0^3}
\frac{\p^2\psi_0}{\p\xi_2^2} +c_{2k}\psi_0\quad
\text{in}\quad\Pi,
\\
\psi_{2k}&=0\quad\text{on}\quad\p\Pi,
\end{aligned}\label{5.59}
\\
&
\begin{aligned}
-\frac{1}{H_0^2}\frac{\p^2\psi_i}{\p\xi_2^2}
-c_0\psi_i=&c_i\psi_0+\frac{\p^2\psi_{i-2k}}{\p\xi_1^2}+c_{2k}\psi_{i-2k}
\\
&- \frac{2H_{2k}\xi_1^{2k}}{(2k)!H_0^3}
\frac{\p^2\psi_{i-2k}}{\p\xi_2^2}+F_i\quad \text{in}\quad\Pi,
\\
\psi_i=&0\quad\text{on}\quad\p\Pi,\quad i\geqslant 2k+1,
\end{aligned}
\label{5.55}
\end{align}
where the functions $\psi_0$, $\psi_i$ are assumed to decay
exponentially as $\xi_1\to\pm\infty$, and
\begin{align*}
F_i=&2\sum\limits_{j=2k+1}^{i} \left(P_{12}^{(j)}(\xi_1)+\xi_2
Q_{12}^{(j)}(\xi_1)\right)\frac{\p^2\psi_{i-j}}{\p\xi_1\p\xi_2}
\\
&+ \sum\limits_{j=2k+1}^{i} \left(P_{22}^{(j)}(\xi_1)+\xi_2
Q_{22}^{(j)}(\xi_1)+\xi_2^2
R_{22}^{(j)}(\xi_1)\right)\frac{\p^2\psi_{i-j}}{\p\xi_2^2}
\\
&+ \sum\limits_{j=2k+2}^{i} \left(P_{2}^{(j)}(\xi_1)+\xi_2
Q_{2}^{(j)}(\xi_1)\right)\frac{\p\psi_{i-j}}{\p\xi_2}+
\sum\limits_{j=2k+1}^{i-1}c_j\psi_{i-j}.
\end{align*}

Problems (\ref{5.54}) can be solved by separation of
variables. It gives the formulas for $\psi_i$,
$i=0,\ldots,2k-1$, and $c_0$,
\begin{equation}\label{5.58}
\psi_i(\xi)=\Psi_i(\xi_1)\sin n\pi\xi_2,\quad c_0=\frac{\pi^2
n^2}{H_0^2},\quad n\in\mathbb{N},
\end{equation}
where $i=0,\ldots,2k-1$, and the functions $\Psi_i$ are to be
determined.  We can consider the problem (\ref{5.59}) as posed
on the interval $(0,1)$ and depending on the parameter $\xi_1$.
Hence, this problem is solvable, if the right hand side is
orthogonal to $\sin n\pi\xi_2$ in $L_2(0,1)$ for each $\xi_1$,
where the inner product is taken with respect to $\xi_2$. Evaluating this
inner product and taking into account (\ref{5.58}), we arrive at
the equation
\begin{equation}\label{5.60}
\left(-\frac{d^2}{d\xi_1^2}- \frac{2\pi^2 n^2
H_{2k}\xi_1^{2k}}{(2k)!H_0^3}\right)\Psi_0= c_{2k}\Psi_0\quad
\text{in}\quad \mathbb{R}.
\end{equation}
The exponential decaying of $\psi_0$ as $|\xi_1|\to\infty$ is
possible, if the same is true for $\Psi_0$. Hence, $\Psi_0$ is
an eigenfunction of $G_n$, and therefore, $c_{2k}$ is the
corresponding eigenvalue. Thus, $\Psi_0=\Phi_{n,m}$, and the
formula (\ref{5.66}) for $c_{2k}$ is valid.

\begin{lemma}\label{lm5.5}
The spectrum of $G_n$
 consists of infinitely many simple positive
isolated eigenvalues $\Lambda_{n,m}$. The associated
eigenfunctions are infinitely differentiable, and decay
exponentially at infinity, namely,
\begin{align*}
&\Phi_{n,m}(\xi_1)=C_\pm
|\xi_1|^{\frac{\l}{2\sqrt{A}}-\frac{1}{2}}\E^{-\frac{\sqrt{A}}{2}|\xi_1|^{2}}(1+o(1)),
&& \text{if}\quad k=1,
\\
&\Phi_{n,m}(\xi_1)=C_\pm
|\xi_1|^{-\frac{k}{2}}\E^{-\frac{\sqrt{A}}{k+1}|\xi_1|^{k+1}}(1+o(1)),
&& \text{if}\quad k>1.
\end{align*}
The last formulas can be differentiated with respect to $\xi_1$.

The equation $(G_n-\L_{n,m})u=f$, $f\in L_2(\mathbb{R})$, is
solvable if and only if
\begin{equation}\label{5.63}
(f,\Phi_{n,m})_{L_2(\mathbb{R})}=0.
\end{equation}
In this case, there is a unique solution $u$ orthogonal to $\Phi_{n,m}$ in
$L_2(\mathbb{R})$. If $f\in C^\infty(\mathbb{R})$ is an
exponentially decaying function satisfying
\begin{equation}\label{5.63a}
f=\Odr(|\xi_1|^\b\exp^{-\frac{\sqrt{A}}{k+1}|\xi_1|^{k+1}}),\quad
|\xi_1|\to\infty,
\end{equation}
and this identity can be differentiated with respect to $\xi_1$, then the
solution $u$ is infinitely differentiable and decays
exponentially,
\begin{equation}\label{5.64}
u=\Odr(|\xi_1|^{\widetilde{\b}}\exp^{-\frac{\sqrt{A}}{k+1}|\xi_1|^{k+1}}),\quad
|\xi_1|\to\infty,
\end{equation}
where $\widetilde{\b}$ is a some number. This identity can be
differentiated with respect to $\xi_1$.
\end{lemma}

\begin{proof}
The statement on the eigenfunctions follows from \cite[Ch. I\!I,
Sec. 2.3, Th. 3.1, Sec. 2.4, Th. 4.6]{Bsh}. The solvability
condition (\ref{5.63}) is due to self-adjointness of $G_n$.
Theorem 4.6 in \cite[Ch. I\!I, Sec. 2.4]{Bsh} gives also the
asymptotic behaviour of the fundamental system of the equation
(\ref{5.60}). Using these formulas and representing the solution
$u$ via Green function, one can easily prove (\ref{5.64}).
\end{proof}

Taking into account (\ref{5.58}), (\ref{5.60}), we can rewrite
(\ref{5.59}) as
\begin{equation*}
-\frac{1}{H_0^2}\frac{\p^2\psi_{2k}}{\p\xi_2^2}-\frac{\pi^2
n^2}{H_0^2}\psi_{2k}=0\quad \text{in}\quad\Pi, \qquad
\psi_{2k}=0\quad\text{on}\quad\p\Pi.
\end{equation*}
Hence, the formula (\ref{5.58}) is valid for $\psi_{2k}$ as well.
In what follows, and for the sake of brevity we will write simply $\L$ and
$\Phi$ instead of $\L_{n,m}$ and $\Psi_{n,m}$.

\begin{lemma}\label{lm5.7}
The problems (\ref{5.54}), (\ref{5.59}), (\ref{5.55}) are
solvable, and their solutions read as follows:
\begin{equation}
\psi_i(\xi)=\widetilde{\psi}_i(\xi)+\Psi_i(\xi_1)\sin\pi
n\xi_2,\quad \widetilde{\psi}_i(\xi)=\sum_{j}
\phi_{i,j}^{(1)}(\xi_1) \phi_{i,j}^{(2)}(\xi_2), \label{5.67}
\end{equation}
where the sum with respect to $j$ is finite, $\phi_{i,j}^{(1)}\in
C^\infty(\mathbb{R})$ are exponentially decaying functions
satisfying (\ref{5.64}), $\phi_{i,j}^{(2)}\in C^\infty[0,\pi]$
are orthogonal to $\sin\pi n\xi_2$ in $L_2(0,\pi)$ for each
$\xi_1\in \mathbb{R}$ and vanish on $\p\Pi$. As $i=0,\ldots,2k$,
the functions $\widetilde{\psi}_i$ are identically zero, while
for $i>2k$ they solve the boundary value problems
\begin{align}
&
\begin{aligned}
-\frac{1}{H_0^2}\frac{\p^2\widetilde{\psi}_i}{\p\xi_2^2}
-\frac{\pi^2 n^2}{H_0^2} &\widetilde{\psi}_i
=\frac{\p^2\widetilde{\psi}_{i-2k}}{\p\xi_1^2} -
\frac{2H_{2k}\xi_1^{2k}}{(2k)!H_0^3}
\frac{\p^2\widetilde{\psi}_{i-2k}}{\p\xi_2^2}
\\
& + \L\widetilde{\psi}_{i-2k}+
\sum\limits_{j=2k+1}^{i-1}c_j\widetilde{\psi}_{i-j}
+\widetilde{F}_i-f_i\sin\pi n\xi_2\quad \text{in}\quad\Pi,
\\
\widetilde{\psi}_i=&0\quad\text{on}\quad\p\Pi,
\end{aligned}\label{5.68}
\\
&
\begin{aligned}
\widetilde{F}_i=&2\sum\limits_{j=2k+1}^{i}
\left(P_{12}^{(j)}+\xi_2
Q_{12}^{(j)}\right)\frac{\p^2\psi_{i-j}}{\p\xi_1\p\xi_2}
\\
&+ \sum\limits_{j=2k+1}^{i} \left(P_{22}^{(j)}+\xi_2
Q_{22}^{(j)}+\xi_2^2
R_{22}^{(j)}\right)\frac{\p^2\psi_{i-j}}{\p\xi_2^2}
\\
&+ \sum\limits_{j=2k+2}^{i} \left(P_{2}^{(j)}+\xi_2
Q_{2}^{(j)}\right)\frac{\p\psi_{i-j}}{\p\xi_2},
\end{aligned}\label{5.71}
\\
&f_i=2(\widetilde{F}_i,\sin\pi n\xi_2)_{L_2(0,\pi)}.\label{5.72}
\end{align}
The functions $\Psi_i\in C^\infty(\mathbb{R})$ satisfy
(\ref{5.64}), are orthogonal to $\Phi$ in $L_2(\mathbb{R})$ and
solve the equations
\begin{equation}\label{5.69}
(G_n-\L)\Psi_{i-2k}=f_{i}+
\sum\limits_{j=2k+1}^{i-1}c_{j}\Psi_{i-j}+c_{i}\Phi.
\end{equation}
The numbers $c_i$, $i\geqslant 2k+1$, are given by the formulas
\begin{equation}\label{5.70}
c_i=-(f_i,\Phi)_{L_2(\mathbb{R})}.
\end{equation}
\end{lemma}

\begin{proof}
We prove the lemma by induction. The statement of the lemma
for $i=0$ follows from (\ref{5.58}). This identity also implies
the formulas (\ref{5.67}) for $i=1,\ldots,2k$, where
$\widetilde{\psi}_i\equiv0$, and $\Psi_i$ are some functions to
be determined.

Assume that the statement of the lemma holds true for $i<p$,
$p\geqslant 2k+1$. Then, in view of (\ref{5.67}), the right hand
side in the equation in (\ref{5.55}) can be rewritten as
\begin{equation}\label{5.73}
\begin{aligned}
&\left(c_p\Phi-(G_n-\L)\Psi_{p-2k}+
\sum\limits_{j=2k+1}^{p-1}c_j\Psi_{p-j} \right)\sin\pi n\xi_2
\\
&+\frac{\p^2\widetilde{\psi}_{p-2k}}{\p\xi_1^2}+\L
\widetilde{\psi}_{p-2k}- \frac{2H_{2k}\xi_1^{2k}}{(2k)!H_0^3}
\frac{\p^2\widetilde{\psi}_{p-2k}}{\p\xi_2^2}+
\sum\limits_{j=2k+1}^{p-1}c_j\widetilde{\psi}_{p-j}
 +\widetilde{F}_p.
\end{aligned}
\end{equation}
The solvability condition of (\ref{5.55}) is the orthogonality
of this right hand side to $\sin\pi n\xi_2$ in $L_2(0,\pi)$ for
each $\xi_1\in\mathbb{R}$. We write this condition, taking into
account the orthogonality of $\widetilde{\psi}_j$, $j<p$, to
$\sin\pi n\xi_2$ in $L_2(0,\pi)$, and the relation
\begin{align*}
\left(\frac{\p^2\widetilde{\psi}_{p-2k}}{\p\xi_2^2},\sin\pi
n\xi_2 \right)_{L_2(0,\pi)}&=
-\pi^2 n^2 \left(\widetilde{\psi}_{p-2k},\sin\pi
n\xi_2\right)_{L_2(0,\pi)}=0.
\end{align*}
This procedure leads us to (\ref{5.69}).  By Lemma~\ref{lm5.5},
the solvability condition of (\ref{5.69}) is exactly the formula
(\ref{5.70}), since the functions $\Psi_{i-2k}$, $i<p$, are
orthogonal to $\Phi$ in $L_2(\mathbb{R})$. We choose the
solution of this equation to be orthogonal to $\Phi$ in
$L_2(\mathbb{R})$ and note that by the formula (\ref{5.67}) for
$\widetilde{\Psi}_i$, $i<p$, the function $f_i$ satisfies
(\ref{5.63a}). Hence, by Lemma~\ref{lm5.5}, the function
$\Psi_{p-2k}$ satisfies (\ref{5.64}). The formulas (\ref{5.67})
and (\ref{5.73}) yield that the right hand side of the equation
in (\ref{5.55}) with $i=p$ can be represented as a finite sum
$\sum\limits_{j} f_{p,j}^{(1)}(\xi_1) f_{p,j}^{(2)}(\xi_2)$,
where $f_{p,j}^{(2)}\in C^\infty[0,1]$ are orthogonal to
$\sin\pi n\xi_2$ in $L_2(0,\pi)$, while the functions
$f_{p,j}^{(1)}\in C^\infty(\mathbb{R})$ satisfy (\ref{5.63a}).
This fact implies the formula (\ref{5.67}) for $i=p$.
\end{proof}

Let us prove the remaining formulas in (\ref{5.66}). It follows from
(\ref{5.71}), (\ref{5.52}) that
\begin{equation}\label{5.74}
\widetilde{F}_{2k+1}=\frac{2\pi n
h_{1}}{H_0}\frac{d\Phi}{d\xi_1} \cos\pi
n\xi_2+\frac{2\pi^2n^2H_{2k+1}\xi_1^{2k+1}}{(2k+1)!H_0^3}
\Phi\sin\pi n\xi_2.
\end{equation}
Therefore,
\begin{equation}\label{5.76}
f_{2k+1}=\frac{2\pi^2n^2H_{2k+1}\xi_1^{2k+1}}{(2k+1)!H_0^3}
\Phi.
\end{equation}
The eigenfunction $\Phi$ is either odd or even with respect to
$\xi_1$, due to the evenness of the potential in the operator
$G_n$. Hence, the function $\Phi^2$ is even, and this is why
$(f_{2k+1},\Phi)_{L_2(\mathbb{R})}=0$, proving the formula for
$c_{2k+1}$.

Employing (\ref{5.52}), by direct calculation we check that
\begin{equation}\label{5.77}
f_{2k+2}=-Q_{12}^{(2k+2)}\frac{d\Phi}{d\xi_1}-\pi^2n^2
P_{22}^{(2k+1)}\Psi_1-\pi^2n^2
P_{22}^{(2k+2)}\Phi-\frac{1}{2}Q_2^{(2k+2)}\Phi.
\end{equation}
By integrating by parts we obtain
\begin{equation*}
\left(\xi_1\frac{d\Phi}{d\xi_1},\Phi\right)_{L_2(\mathbb{R})}=\frac{1}{2}\int\limits_{\mathbb{R}}
\xi_1\di\Phi^2=-\frac{1}{2}.
\end{equation*}
Now the formula (\ref{5.66}) for $c_{2k+2}$ follows from two
last identities and (\ref{5.52}).

We proceed to the justification of the expansions (\ref{5.3a}).
For any $N>7(k+1)$ we denote
\begin{equation*}
\psi_{\e,N}(x):=\chi(x_1) \sum\limits_{i=0}^{N}
\eta^i\psi_i(\xi),
\quad\mu_{\e,N}:=c_0+\sum\limits_{i=2k}^{N}\eta^i c_i,\quad
\l_{\e,N}:=\e^{-2}\mu_{\e,N}.
\end{equation*}

\begin{lemma}\label{lm5.2}
The pair $\psi_{\e,N}$, $\l_{\e,N}$ satisfies the boundary value
problem
\begin{equation}\label{5.29}
-\D\psi_{\e,N}=\l_{\e,N}\psi_{\e,N}+g_{\e,N}\quad\text{in}\quad
\Om_\e,\qquad\psi_{\e,N}=0\quad\text{on}\quad\p\Om_\e,
\end{equation}
and
\begin{equation*}
\|g_{\e,N}\|_{L_2(\Om_\e)}=\Odr(\e^{\frac{N}{k+1}-3}). 
\end{equation*}
\end{lemma}

This lemma follows from the definition of the problems
(\ref{5.54}), (\ref{5.55}), Lemma~\ref{lm5.7}, and the
exponential decaying of $\psi_i$ as $|\xi_1|\to\infty$.

We can rewrite the problem (\ref{5.29}) as
\begin{equation}\label{5.30}
\psi_{\e,N}=(\l_{\e,N}+1)A_\e\psi_{\e,N}+A_\e g_{\e,N},
\end{equation}
where $A_\e$ indicates the inverse of $-\D_{\Om_\e}^{D}+1$. It
is clear that the operator $A_\e$ is self-adjoint, has a compact
resolvent, and satisfies the estimate
$\|A_\e\|\leqslant 1$, uniformly in $\e$. Now we rewrite (\ref{5.30}) as
\begin{equation*}
\frac{1}{\l_{\e,N}+1}\psi_{\e,N}=A_\e\psi_{\e,N}+
\frac{1}{\l_{\e,N}+1}A_\e g_{\e,N},
\end{equation*}
and apply Lemma~1.1 from \cite[Ch. I\!I\!I, Sec. 1.1]{IOS}. This
lemma yields that given $N$ for each $\e$ there exists an
eigenvalue $\tau_N(\e)$ 
of $A_\e$ such that
\begin{equation*}
\left|\tau_N(\e)-(1+\l_{\e,N})^{-1}\right|=
\Odr(\e^{\frac{N}{k+1}-3}).
\end{equation*}
The associated eigenfunction $\widetilde{\psi}_{\e,N}$ satisfies
the relation
\begin{equation}\label{5.32}
\|\widetilde{\psi}_{\e,N}-\psi_{\e,N}\|_{L_2(\Om_\e)}=
\Odr(\e^{\frac{N}{k+1}-3}).
\end{equation}
Hence, there exists an eigenvalue $\widetilde{\l}_N(\e)$ of
$-\D^{D}_{\Om_\e}$ such that
\begin{equation}\label{5.33}
|\widetilde{\l}_N(\e)-\l_{\e,N}|\leqslant
C_N\e^{\frac{N}{k+1}-7},
\end{equation}
where $C_N$ are some constants.

Let $\e_N$ be a monotone sequence such that
$C_N\e^{1/(k+1)}\leqslant C_{N-1}$ as $\e\leqslant \e_N$. Given
$n$ and $m$ corresponding to the series (\ref{5.3a}) (see
(\ref{5.58}), (\ref{5.60})), we chose the eigenvalue
$\l_{n,m}(\e)$, letting $\l_{n,m}(\e):=\l_N(\e)$ as
$\e\in[\e_N,\e_{N+1})$. Employing (\ref{5.33}), one can check
easily that the eigenvalue $\l_{n,m}$ satisfies (\ref{5.65}).



We denote $\z_1:=x_1$, $\z_2:=x_2\e^{-1}$, $\Om:=\{\z:
\z_1\in(0,1),-{h}_-(\z_1)<\z_2<{h}_+(\z_1)\}$,
$\Om^\d:=\Om\cap\{\z: |\z_1-\xx|<\d\}$,
\begin{equation}\label{5.33a}
\widehat{\psi}_N(\z):=\widetilde{\psi}_{\e,N}(x)-\psi_{\e,N}(x).
\end{equation}

\begin{lemma}\label{lm5.4}
The relations
\begin{equation*}
\|\nabla \widehat{\psi}_N\|_{L_2(\Om_\e)}=
\Odr(\e^{\frac{N}{2(k+1)}-\frac{7}{2}}),\quad
\|\widehat{\psi}_N\|_{\H^2(\Om^{\d/3})}
=\Odr(\e^{\frac{N}{2(k+1)}-\frac{7}{2}})
\end{equation*}
hold true.
\end{lemma}

\begin{proof}
Employing Lemma~\ref{lm5.2} and (\ref{5.32}), (\ref{5.33}), by
integrating by parts one can easily check that
\begin{equation}\label{5.38}
\|\nabla \widehat{\psi}_N\|_{L_2(\Om_\e)}=
\Odr(\e^{\frac{N}{2(k+1)}-\frac{7}{2}}).
\end{equation}
It follows from Lemma~\ref{lm5.2}, the equation for
$\psi_{\e,N}$ and (\ref{5.32}), (\ref{5.33}), (\ref{5.38}) that
the function
$\phi(\z):=\chi(x_1)\widetilde{\psi}_{\e,N}(x)-\psi_{\e,N}(x)$
satisfies the boundary value problem
\begin{equation}\label{5.39}
-\left(\e^2\frac{\p^2}{\p\z_1^2}+\frac{\p^2}{\p\z_2^2}\right)
\phi=\widehat{g}\quad\text{in}\quad
\Om^{\d},\qquad\widehat{\psi}_N=0\quad\text{on} \quad\p\Om^{\d},
\end{equation}
where
\begin{equation}\label{5.40}
\|\widehat{g}\|_{L_2(\Om^{\d})}=\Odr(\e^{\frac{N}{k+1}-\frac{9}{2}}),
\quad \|\nabla
\phi\|_{L_2(\Om^{\d})}=\Odr(\e^{\frac{N}{2(k+1)}-3}).
\end{equation}
Integrating by parts in the same way as in the proof of
Lemma~7.1 in \cite[Ch.~3, Sec.~7]{Ld}, one can check that
\begin{equation}\label{5.41}
\int\limits_{\Om^{\d}}\frac{\p^2 \phi}{\p\z_1^2} \frac{\p^2
\phi}{\p\z_2^2}\di\z=\int\limits_{\G^{\d}} \frac{\p
\phi}{\p\z_1}\left(\nu_1\frac{\p^2 \phi}{\p\z_2^2}-
\nu_2\frac{\p^2 \phi}{\p\z_1\p\z_2}\right)\di
s+\int\limits_{\Om^{\d}} \left(\frac{\p^2
\phi}{\p\z_1\p\z_2}\right)^2\di\z.
\end{equation}
Here $\G^{\d}:=\p\Om^{\d}\setminus\{\z: \z_1=-\d\}\setminus\{\z:
\z_1=\d\}$, $\nu=\nu(s)$, $\nu=(\nu_1,\nu_2)$ is the outward
normal to $\G^{\d}$ and $s$ is the arc length of $\G^{\d}$.
Employing the identity
\begin{equation*}
\frac{\p \phi}{\p s}=0\quad\text{on}\quad\G^{\d},
\end{equation*}
we continue the calculations
\begin{align*}
&\int\limits_{\G^{\d}} \frac{\p
\phi}{\p\z_1}\left(\nu_1\frac{\p^2 \phi}{\p\z_2^2}-
\nu_2\frac{\p^2 \phi}{\p\z_1\p\z_2}\right)\di
s=\int\limits_{\G^{\d}} \nu_1\frac{\p \phi}{\p\nu}\frac{\p}{\p
s}\frac{\p \phi}{\p x_2}\di s=\int\limits_{\G^{\d}}
\nu_1\frac{\p \phi}{\p\nu}\frac{\p}{\p s}\nu_2\frac{\p
\phi}{\p\nu}\di s
\\
&=\int\limits_{\G^{\d}} \nu_1\nu_2'\left(\frac{\p
\phi}{\p\nu}\right)^2\di s+\frac{1}{2}\int\limits_{\G^{\d}}
\nu_1\nu_2\frac{\p}{\p s}\left(\frac{\p \phi}{\p\nu}\right)^2\di
s =\frac{1}{2}\int\limits_{\G^{\d}}
(\nu_1\nu_2'-\nu_1'\nu_2)\left(\frac{\p \phi}{\p\nu}\right)^2\di
s
\end{align*}
where $'$ denotes the derivative with respect to $s$. The obtained
formula implies that
\begin{align*}
\left|\int\limits_{\G^{\d}} \frac{\p
\phi}{\p\z_1}\left(\nu_1\frac{\p^2 \phi}{\p\z_2^2}-
\nu_2\frac{\p^2 \phi}{\p\z_1\p\z_2}\right)\di s\right|\leqslant
\e^3 \|\phi\|_{\H^2(\Om^{\d})}^2+C\e^{-3}
\|\phi\|_{\H^1(\Om^{\d})}^2,
\end{align*}
where the constant $C$ is independent of $\e$. Hence, by
(\ref{5.41}),
\begin{align*}
&\left(\frac{\p^2 \phi}{\p\z_1^2}, \frac{\p^2
\phi}{\p\z_2^2}\right)_{L_2(\Om^{\d})}\geqslant \Big\|\frac{\p^2
\phi}{\p\z_1\p\z_2}\Big\|^2-\e^3
\|\phi_{\z\z}\|_{L_2(\Om^{\d})}^2-C\e^{-3}
\|\phi\|_{\H^1(\Om^{\d})}^2,
\\
&\|\phi_{\z\z}\|_{L_2(\Om^{\d})}^2:=
\Big\|\frac{\p^2\phi}{\p\z_1^2}
\Big\|_{L_2(\Om^{\d})}^2+\Big\|\frac{\p^2\phi}{\p\z_2^2}
\Big\|_{L_2(\Om^{\d})}^2+\Big\|\frac{\p^2\phi} {\p\z_1\p\z_2}
\Big\|_{L_2(\Om^{\d})}^2.
\end{align*}
Employing this estimate, we obtain
\begin{align*}
&\|\widehat{g}\|_{L_2(\Om^{\d})}^2=\e^4\Big\|
\frac{\p^2\phi}{\p\z_1^2}\Big\|_{L_2(\Om^{\d})}^2+2\e^2
\left(\frac{\p^2 \phi}{\p\z_1^2}, \frac{\p^2
\phi}{\p\z_2^2}\right)_{L_2(\Om^{\d})}+ \Big\|
\frac{\p^2\phi}{\p\z_2^2}\Big\|_{L_2(\Om^{\d})}^2
\\
&\geqslant \e^4(1-2\e) \|\phi_{\z\z}\|_{L_2(\Om^{\d})}^2-
C\e^{-1}\|\phi\|_{L_2(\Om^{\d})}^2.
\end{align*}
Combining this estimate with (\ref{5.40}), we complete the
proof.
\end{proof}

The proven lemma implies the asymptotics (\ref{5.3}).

\section{Proof of Theorem~\ref{th5.2}\label{proof2}}

Let us prove first that the eigenvalue $\l_{1,1}(\e)$ is the
ground state. To do this, it is sufficient to show that the
associated eigenfunction $\widetilde{\psi}_{\e}$ is
non-negative. First we observe that by standard smoothness
improving theorems for the solutions of elliptic boundary value
problems (see, for instance, \cite[Ch. I\!V, Sec. 2]{M}) the
belonging $\widetilde{\psi}_{\e}\in C^\infty(\Om)\cap
C^\infty(\overline{\Om^{5\d/6}})$ holds true.

\begin{lemma}\label{lm3.5}
The relations
\begin{equation*}
\|\widehat{\psi}_N\|_{\H^3(\Om^{\d/4})}=\Odr(\e^{M_1})
\end{equation*}
hold true, where $M_1=M_1(N)\to+\infty$ as $N\to+\infty$.
\end{lemma}

\begin{proof}
We denote
\begin{equation*}
\widetilde{\phi}(\z):=\chi(\z_1/2)\widetilde{\psi}_{\e}(x)-
\psi_{\e,N}(x).
\end{equation*}
It follows from Lemmas~\ref{lm5.2},~\ref{lm5.4}, the equation
for $\psi_{\e,N}$ and (\ref{5.32}), (\ref{5.33}), (\ref{5.38})
that the function $\widetilde{\phi}$ satisfies the boundary
value problem (\ref{5.39}), where the right hand side
$\widehat{g}$ obeys the estimate
\begin{equation*}
\|\widehat{g}\|_{\H^1(\Om^\d)}=\Odr(\e^{M_2}).
\end{equation*}
Here and throughout the proof by $M_i$ we denote some functions
$M_i(N)$ such that $M_i(N)\to+\infty$ as $N\to+\infty$. 

In the problem (\ref{5.39}) we pass to the variables
$(\z_1,\xi_2)$. Then we differentiate the obtained problem
with respect to $\z_1$, that leads us to the boundary value problem
\begin{gather*}
\left(\eta^2 K_{11} \frac{\p^2}{\p\z_1^2}+2\eta
K_{12}\frac{\p^2}{\p\z_1\p\xi_2}
+K_{22}\frac{\p^2}{\p\xi_2^2}+K_{2}\frac{\p}{\p\xi_2}
\right)\widetilde{\phi}_1=\widehat{g}_1 \quad\text{in}\quad
\Pi(\d),
\\
\widetilde{\phi}_1=0\quad\text{on}\quad\p\Pi(\d),
\end{gather*}
where $\Pi(\d):=\{(\z_1,\xi_2): |\z_1-\xx|<\d, 0<\xi_2<1\}$,
$\widetilde{\phi}_1:=\frac{\p\widetilde{\phi}}{\p\z_1}$,
$\|\widehat{g}_1\|_{L_2(\Pi(\d))}=\Odr(\e^{M_4})$,
$K_{ij}=K_{ij}(\z_1,\xi_2,\eta)$,
$K_{2}=K_{2}(\z_1,\xi_2,\eta)$, and we remind that these
functions were defined in (\ref{5.53a})

 Now we return
back to the variables $z$ and obtain the problem for
$\widetilde{\phi}_1$, which is similar to (\ref{5.39}). It
allows us to apply the apriori estimate established in the proof
of Lemma~\ref{lm5.4} and obtain,
\begin{equation*}
\|\widetilde{\phi}_1\|_{\H^2(\Pi(\d))}=\Odr(\e^{M_5}).
\end{equation*}
Hence,
\begin{equation*}
\Big\|\frac{\p^3\widetilde{\phi}}{\p\z_1^3}\Big\|_{L_2(\Pi(\d))}
+\Big\|\frac{\p^3\widetilde{\phi}}{\p\z_1^2\p\xi_2}\Big\|_{L_2(\Pi(\d))}
\Big\|\frac{\p^3\widetilde{\phi}}{\p\z_1\xi_2^2}\Big\|_{L_2(\Pi(\d))}
=\Odr(\e^{M_6}).
\end{equation*}
To obtain the similar estimate for
$\big\|\frac{\p^3\widetilde{\phi}}{\p\xi_2^3}\big\|_{L_2(\Pi(\d))}$,
it is sufficient to write the equation for $\widetilde{\phi}$ in
the variables $(\z_1,\xi_2)$ and differentiated it with respect to
$\xi_2$. The obtained equation allows one to express
$\frac{\p^3\widetilde{\phi}}{\p\xi_2^3}$ in terms of other
derivatives of this function, and as a result to obtain the
estimate for
$\big\|\frac{\p^3\widetilde{\phi}}{\p\xi_2^3}\big\|_{L_2(\Pi(\d))}$.
\end{proof}

By standard embedding theorems it follows from
Lemmas~\ref{lm5.4} and~\ref{lm3.5} that
\begin{equation}\label{5.42}
\|\widehat{\psi}_N\|_{C^1(\overline{\Om^{\d/4}})}
=\Odr(\e^{M_5}),
\end{equation}
where the function $\widehat{\psi}_N$ was introduced in
(\ref{5.33a}), $\widetilde{\psi}_\e:=\widetilde{\psi}_{\e,N}$,
and $M_5=M_5(N)\to+\infty$ as $N\to+\infty$.

Since the function $\Phi_{1,1}$ is the ground state of $G_1$, it
is strictly positive. Taking  into account this fact and also
the definition (\ref{5.58}) of $\psi_0$, we conclude that there
exists a positive constant $C$ independent of $\e$ such that
\begin{equation*}
\frac{\p\psi_{\e,N}}{\p\xi_2}\Big|_{\z_2=-{h}_-(\z_1)}>C, \quad
\frac{\p\psi_{\e,N}}{\p\xi_2}\Big|_{\z_2={h}_+(\z_1)}<C \quad
\text{as}\quad |\z_1-\xx|\leqslant 2\a\eta.
\end{equation*}
We also observe that $\psi_0$ is non-negative as
$|\z_1-\xx|\leqslant 2\a\eta$, $-{h}_-(\z_1)\leqslant
\z_2\leqslant {h}_+(\z_2)$. Two last facts and (\ref{5.42})
imply that for any fixed number $\a>0$ the estimate
\begin{equation*}
\widetilde{\psi}_\e\geqslant 0\quad \text{as}\quad
|\z_1-\xx|\leqslant 2\a\eta,\quad -{h}_-(\z_1)\leqslant
\z_2\leqslant {h}_+(\z_1)
\end{equation*}
holds true, if $\e$ is small enough.

Let us show that in the remaining part of $\Om$ the function
$\widetilde{\psi}_\e$ is also non-negative. In order to do it,
we employ the maximum principle. We can not directly apply this
principle to the equation for $\widetilde{\psi}_\e$, since
$\mu_\e>0$, while the potential should be negative. This is why
we first make a special change of the function to make the
potential negative and then we apply the maximum principle.

Consider the equation for $\mathcal{\psi}_\e$ in $\Om\cap\{\z:
\xx-\d/2<z_1<\xx-2\a\eta\}$ and pass to the variables
$(\z_1,\xi_2)$. It leads to the equation
\begin{equation*}
\left(\eta^2 K_{11} \frac{\p^2}{\p\z_1^2}+2\eta
K_{12}\frac{\p^2}{\p\z_1\p\xi_2}
+K_{22}\frac{\p^2}{\p\xi_2^2}+K_{2}\frac{\p}{\p\xi_2}+
\widetilde{\mu}_\e
\right)\psi_\e=0\quad\text{in}\quad\Pi(\d,\a),
\end{equation*}
where $\Pi(\d,\a):=\{(\z_1,\xi_2): \d/2<\z_1<\xx-\a\eta\}$,
$K_{ij}=K_{ij}(\z_1-\xx,\xi_2,\eta)$,
$K_2=K_2(\z_1-\xx,\xi_2,\eta)$.

We introduce the function
\begin{equation*}
A(\z_1,\xi_2,\eta):=-\frac{1}{\eta^{2k+2}}
\int\limits_0^{\xi_2}\frac{K_2(\z_1,t,\eta)}{2
K_{22}(\z_1,\xi_2,\eta)}\di t.
\end{equation*}
Although this function can be calculated explicitly, it is not
essential for our further calculations. We just mention that
this function is infinitely differentiable for the considered
values of $\z_1$, $\xi_2$, and is bounded uniformly in $\z_1$,
$\xi_2$, $\eta$. The latter is true for all the derivatives of
$A$ with respect to $\z_1$ and $\xi_2$.

Now we make the change of $\widetilde{\psi}_\e$,
\begin{align*}
u(\z_1,\xi_2)&:=
\widetilde{\psi}_{\e}(\z_1,\xi_2)\Ups_1^{-1}(\z_1,\xi_2),
\\
\Ups_1(\z_1,\xi_2,\eta)&:=\E^{\eta^{2k+2}A(\z_1,\xi_2,\eta)}
\sin\frac{\pi(\z_2+\b\a^{2k}\eta^{2k})} {1+2\b\a^{2k}\eta^{2k}},
\end{align*}
where $\b>0$ is a constant to be determined below. The
equation for $u$ reads as follows,
\begin{equation}\label{6.11}
\begin{aligned}
\Bigg(&\eta^2 K_{11} \frac{\p^2}{\p\z_1^2}+2\eta
K_{12}\frac{\p^2}{\p\z_1\p\xi_2} +K_{22}\frac{\p^2}{\p\xi_2^2}
+A_1\frac{\p}{\p\z_1}+
A_2\frac{\p}{\p\xi_2}+V_1\Bigg)u=0\quad\text{in}\quad\Pi(\d,\a),
\\
&V_1:=-\frac{\pi^2}{(1+2\b\a^{2k}\eta^{2k})H^2(\z_1)}
+\widetilde{\mu}_\e+A_0,
\end{aligned}
\end{equation}
where $A_i$ are some smooth functions, and the estimate
\begin{equation*}
|A_0|\leqslant C\eta^{2k+2}
\end{equation*}
holds true uniformly in $(\z_1,\xi_2)\in\Pi(\d,\a)$ and $\eta$.
Employing this estimate and the asymptotic for $\l_{1,1}(\e)$
proven in Theorem~\ref{th5.1}, we obtain
\begin{equation}\label{6.17}
\begin{aligned}
V_1& \leqslant
-\frac{\pi^2}{(1+2\b\a^{2k}\eta^{2k})h^2(\xx-2\a\eta)}
+\widetilde{\mu}_\e+A_0
\\
&\leqslant \eta^{2k}\left( \a^{2k}\left(\frac{2^{2k+1} \pi^2
H_{2k}}{(2k)! H_0^3}+2\pi^2\b\right)+\L_{1,1}
\right)+C_5\eta^{2k+1}.
\end{aligned}
\end{equation}
Hence, if we choose $\b$ small enough and $\a$ large enough, we
can claim that
\begin{equation*}
V_1<0
\end{equation*}
in $\Pi(\d)$, if $\eta$ is small enough.

We proceed now to the domain $\Om_-:=\Om\cap\{\z:
\z_1<\xx-\d/4\}$. We denote $C_6:=H_0-H(\xx-\d/4)>0$. It is
clear that there exists a function $\widetilde{h}\in
C^\infty[0,\xx-\d/4]$ such that
\begin{equation*}
\frac{C_6}{6}\leqslant \widetilde{h}(\z_1)-{h}_-(\z_1)\leqslant
\frac{C_6}{5}\quad\text{as}\quad 0\leqslant \z_1\leqslant
\xx-\d/4.
\end{equation*}
Now we introduce the function $u$ as
\begin{equation}\label{5.37}
u(\z):= \widetilde{\psi}_{\e}(\z)\Ups_2^{-1}(\z_1,\xi_2), \quad
\Ups_2(\z):=\sin\frac{\pi(\z_2+\widetilde{h}(\z_1))}
{H(\xx-\d/4)+C_6/2}.
\end{equation}
This function solves the equation
\begin{align*}
\Bigg(&\e^2\frac{\p^2}{\p\z_1^2}+\frac{\p^2}{\p\z_2^2}+
2\e^2\Ups_2^{-1}\frac{\p\Ups_2}{\p\z_1}+
2\Ups_2^{-1}\frac{\p\Ups_2}{\p\z_2}+V_2\Bigg)
u=0\quad\text{in}\quad\Om_-,
\\
&V_2:=\widetilde{\mu}_{\e}-
\frac{\pi^2}{\big(H(\xx-d/4)+C_6/2\big)^2}+\e^2\Ups_2^{-1}
\frac{\p^2\Ups_2} {\p\z_1^2}.
\end{align*}
Employing the asymptotic (\ref{5.65}), (\ref{5.66}), we see that
\begin{equation}\label{6.19}
V_2\leqslant
\frac{\pi^2}{H_0^2}-\frac{\pi^2}{\big(H(\xx-\d/4)+C_6/2\big)^2}
+C_7\e^2<C_8<0,
\end{equation}
if $\e$ is small enough, where the constant $C_8$ is independent
of $\e$.

Now we denote
\begin{equation*}
\Ups(\z):=\Ups_1(\z_1,\xi_2)\chi(\z_1)+
\Ups_2(\z)\big(1-\chi(\z_1)\big),
\end{equation*}
and introduce the function $u$ by (\ref{5.37}) where $\Ups_2$ is
replaced by $\Ups$. One can write the equation for $u$; we need
just to check the sign of the potential. As $\z_1\leqslant
\xx-\d/2$, this potential is that in (\ref{6.11}) and we have
already shown that its negative. As $\z_1\geqslant \xx-\d/3$,
the potential is also negative that is due to (\ref{6.17}). As
$\xx-\d/2<\z_1<\xx-\d/3$, the potential reads as follows,
\begin{equation*}
V:=\Ups^{-1}\left(\Ups_1 V_1+\Ups_2
V_2+\e^2\left(2\chi'\left(\frac{\p\Ups_1}{\p\z_1}-
\frac{\p\Ups_2}{\p\z_1}\right)+\chi''(\Ups_1-\Ups_2)\right)\right).
\end{equation*}
The first two terms in the brackets are negative, as it was
shown above. The last term can be estimated uniformly in $\z$
and $\e$ as
\begin{equation*}
\left|2\chi'\left(\frac{\p\Ups_1}{\p\z_1}-
\frac{\p\Ups_2}{\p\z_1}\right)+\chi''(\Ups_1-\Ups_2)\right|\leqslant
C_9.
\end{equation*}
In view of (\ref{6.19}) we conclude now that the potential $V$
is negative as $\xx-\d/2<\z_1<\xx-\d/3$, if $\e$ is small
enough. Therefore, we can apply the maximum principle to the
function $u$ introduced via $\Ups$. This function vanishes on
$\p\Om\cap\{\z: \z_1<\xx-\a\eta\}$ and is non-negative as
$\z_1=\xx-\a\eta$. Hence, it is non-negative inside
$\Om\cap\{\z: \z_1<\xx-\a\eta\}$ and the same is true for
$\widetilde{\psi}_\e$. In the same way one can also check that
this function is non-negative in $\Om\cap\{\z:
\z_1<\xx-\a\eta\}$. It implies that the eigenvalue
$\l_{1,1}(\e)$ is the lowest one.

Let us prove the formulas (\ref{5.2}). In the case considered
$n=m=k=1$. The operator $G_1$ is the harmonic oscillator, and
its eigenvalues and eigenfunctions are known explicitly. Namely,
\begin{equation}\label{6.1}
\L_{1,1}=\tht,\quad \Phi_{1,1}(\xi_1)=\frac{\tht^{1/4}\E^{-\tht
\xi_1^2/2}}{\pi^{1/4}}, \quad
\tht:=\frac{\pi\big(-H_2\big)^{1/2}}{H_0^{3/2}}.
\end{equation}
The first identity proves the formula for $c_2^{(1,1)}$. The
formula for $c_3^{(1,1)}$ follows from (\ref{5.66}).

It is easy to check by direct calculation that
\begin{equation}\label{6.2}
\begin{aligned}
&\Psi_1^{(1,1)}=\frac{1}{18}\frac{\pi^2 H_3}{\tht^2 H_0^3}
\xi_1(\tht\xi_1^2+3)\Phi_{1,1},
\\
&\|\xi_1^2\Phi_{1,1}\|_{L_2(\mathbb{R})}^2=\frac{3}{4\tht^2},
\quad (\xi_1^3\Psi_{1}^{(1,1)},\Phi_{1,1})_{L_2(\mathbb{R})}=
\frac{11}{48}\frac{\pi^2 H_3}{\tht^4 H_0^3}.
\end{aligned}
\end{equation}
Now we substitute these identities into formula (\ref{5.66})
for $c_{2k+2}$ and arrive at formula (\ref{5.2}) for
$c_4^{(1,1)}$.

To proceed further, we need the formulas for the coefficients of
the series (\ref{5.51}) up to the order $\eta^6$. They read as
follows,
\begin{equation}\label{6.3}
\begin{aligned}
&P_{12}^{(5)}=\frac{h_{3}H_0-h_{1}H_2} {2H_0^2}\xi_1^2, \quad
P_{12}^{(6)} =\frac{h_{4}H_0-h_{1}H_3-3h_{2}H_2}
{6H_0^2}\xi_1^3,
\\
&Q_{12}^{(5)}:=-\frac{H_3}{2H_0}\xi_1^2,\quad
Q_{12}^{(6)}:=-\frac{H_4H_0- 3H_2^2}{6H_0^2}\xi_1^3,
\\
&P_{22}^{(5)}:=\frac{2h_{1}h_{2}}{H_0^2}\xi_1
+\frac{H_2H_3}{2H_0^4}\xi_1^5-\frac{H_5}{60 H_0^3 }\xi_1^5,
\\
&P_{22}^{(6)}:=\frac{45H_4H_2H_0 -180H_2^3-H_6H_0^2+30H_3^2H_0}
{360H_0^5}\xi_1^6+\frac{h_{3}h_{1}H_0-h_{1}^2H_2+
h_{2}^2H_0}{H_0^3}\xi_1^2,
\\
&Q_{22}^{(5)}=-\frac{2h_{1}H_2\xi_1}{H_0^2},\quad
Q_{22}^{(6)}=-\frac{h_{1}H_3+2h_{2}H_2}{H_0^2}\xi_1^2, \quad
 R_{22}^{(5)}=0,\quad R_{22}^{(6)}=\frac{H_2^2\xi_1^2}{H_0^2},
\\
&P_2^{(5)}=\frac{h_{3}H_0-2h_{1}H_2}{H_0^2}\xi_1, \quad
P_2^{(6)}=\frac{h_{4}H_0-5h_{2}H_2 -2h_{1}H_3}{2H_0^2}\xi_1^2,
\\
&Q_2^{(5)}=-\frac{H_3}{H_0}\xi_1,\quad Q_2^{(6)}=
\frac{5H_2^2-H_4 H_0} {2H_0^2}\xi_1^2.
\end{aligned}
\end{equation}
It follows from (\ref{5.68}), (\ref{5.74}), (\ref{5.76}) that
\begin{equation}\label{6.4}
\widetilde{\psi}_3^{(1,1)}(\xi)=\frac{1}{2}H_0h_{1}(1-2\xi_2)
\sin\pi\xi_2.
\end{equation}
In view of (\ref{5.77}), (\ref{5.52}) we can also find the
function $\Psi_2^{(1,1)}$ explicitly,
\begin{equation}\label{6.5}
\begin{aligned}
\Psi_2^{(1,1)}=\bigg(&-\frac{\pi^2 H_3^2}{648 H_0^3H_2 }\xi_1^6
+\frac{\pi^2\xi_1^4}{864H_0^4H_2\tht} \Big(9H_4H_2H_0
-11H_3^2H_0-81H_2^3\Big)
\\
&+\frac{9H_2^3-9H_4H_2H_0 +11H_3^2H_0}{288H_0H_2^2} \xi_1^2 +
\frac{3H_4}{128 H_2\tht}-\frac{109 H_3^2 }{3456 H_2^2\tht} -
\frac{11H_2}{128H_0\tht}\bigg)\Phi_{1,1}.
\end{aligned}
\end{equation}
Now we use the formulas (\ref{5.2}) for $c_i^{(1,1)}$,
$i\leqslant 4$, and (\ref{5.52}), (\ref{5.58}), (\ref{5.71}),
(\ref{5.72}), (\ref{6.1}), (\ref{6.2}), (\ref{6.3}),
(\ref{6.4}), (\ref{6.5}) and obtain
\begin{equation}\label{6.6}
\begin{aligned}
f_5=\Bigg(& -\frac{\pi^4H_3^3\xi_1^9}{1944H_2H_0^6}
-\frac{\pi^4H_3\xi_1^7 }{2592H_2H_0^7\tht}\Big(-21H_4H_0H_2
+189H_2^3+11H_3^2H_0\Big)
\\
&-\frac{\pi^2\xi_1^5} {4320H_2^2H_0^4} \Big(105H_5H_3H_2H_0
+1815H_2^3H_3-72H_2^2 H_5H_0-55H_3^3H_0\Big)
\\
&+\frac{\xi_1^3\pi^2H_3}{10368H_0^5H_2^2\tht} \Big(
5175H_2^3H_0-576\pi^2 h_{1}^2
h2^2+81H_4H_2H_0^2-109H_3^2H_0^2\Big)
\\
&+\frac{\xi_1}{12 H_2H_0^2}\Big(
2\pi^2H_3h_{1}^2-24\pi^2H_2h_{1}h_{2}+3H_3H_2H_0+12\pi^2H_2^2
h_{1}\Big) \Bigg)\Phi_{1,1}.
\end{aligned}
\end{equation}
The function $f_5$ is odd with respect to $\xi_1$, and
therefore, in view of (\ref{5.70}) and evenness of $\Phi_{1,1}$,
the formula (\ref{5.2}) is valid for $c_5^{(1,1)}$. Employing
the obtained identity for $f_5$, we can solve explicitly the
equation (\ref{5.69}) for $\Psi_3^{(1,1)}$:
\begin{align*}
\Psi_3^{(1,1)}=\Bigg(& -\frac{\pi^4H_3^3}{34992H_0^6\tht}\xi_1^9
+ \frac{\pi^2\xi_1^7 H_3}{15552 H_0^4 H_2^2}\Big(7H_3^2H_0
+81H_2^3-9H_4H_2H_0\Big)
\\
&+\frac{\pi^2\xi_1^5}{64800H_0^4H_2^2\tht}
\Big(205H_3^3H_0-315H_4H_3H_2H_0-1305H_3H_2^3+108H_5H_2^2H_0\Big)
\\
& +\frac{\xi_1^3}{311040H_0H_2^3} \Big(4455H_4H_3H_2H_0
-2515H_3^3H_0-3375H_3H_2^3-1728H_2^2H_5H_0\Big)
\\
&+\frac{\xi}{103680 H_0^2H_2^3\tht}
\Big(-1855H_3^3H_0^2+3915H_4H_3H_2H_0^2+14445H_3H_2^3H_0
\\
&-1728H_5H_2^2H_0^2-103680\pi^2H_2^3h_{1} h_{2}
+51840\pi^2H_2^4h_{1}\Big) \Bigg)\Phi_{1,1}
\end{align*}
We substitute the relation obtained, the formulas (\ref{5.2})
for $c_i^{(1,1)}$, $i\leqslant 5$, (\ref{5.52}), (\ref{5.58}),
(\ref{5.71}), (\ref{5.72}), (\ref{6.1}), (\ref{6.2}),
(\ref{6.3}), (\ref{6.4}), (\ref{6.5}), (\ref{6.6}) into
(\ref{5.70}) and arrive at the formula (\ref{5.2}) for
$c_6^{(1,1)}$. The proof is complete.

\section{Examples\label{examp}}

We shall now apply our results to obtain the expansion of the first
eigenvalue for different domains, and compare the values obtained with
a numerical approximation. We are indebted to Pedro Antunes for carrying
out the necessary numerical computations.

We have chosen five examples illustrating several possibilities for the
functions $H$, $h_{+}$ and $h_{-}$. The first three correspond to algebraic
curves, namely the circle, the lemniscate and the {\it bean} curve. For these
$h_{+} = h_{-}$, and the fit between the four--term asymptotic expansion
and the numerical approximation is very good: for the lemniscate and the
bean curve the error is always below $2\%$ and for the disk the maximum
error is around $5\%$.

The last two examples have $h_{+} \neq h_{-}$, and the second one is nonconvex.
Here we see that the error becomes much larger -- in the last example it can go up
to $50\%$.

\begin{examp}[Disk]
{\rm Consider the disk centred at $(1/2,0)$ and radius $1/2$ for which we
have $H(x_{1})=2h_{+}(x_{1})=2h_{-}(x_{1})=2(x-x^{2})^{1/2}$. The maximum
of $H$ occurs at $x_{1}=1/2$ and we obtain the expansion given
by~(\ref{diskasympt}). Comparing these results with those of Daymond referred
to in~\cite{jose} we see that the error up to $\e$ equals one is maximal at
one and is around five per cent.
}
\end{examp}

\begin{examp}[Lemniscate]\label{ex:lemn}
{\rm  Consider the lemniscate defined by
\[
\left( x_{1}^2+ x_{2}^2\right)^2 = x_{1}^2-x_{2}^2.
\]
In this case we have
\[
H(x_{1})=2h_{+}(x_{1})=2h_{-}(x_{1}) = 2\left[-\frac{\ds 1}{\ds 2}-x_{1}^2+
\frac{\ds 1}{\ds 2}\left( 1+8x_{1}^2\right)^{1/2}\right]^{1/2},
\]
and the maximum of $H$ is now situated at $\sqrt{3}/(2\sqrt{2})$. This yields
\[
\lambda_{1}(\e) = \frac{\ds 2\pi^{2}}{\ds
\e^2}+\frac{\ds 2\sqrt{3}\pi}{\ds \e} +\frac{\ds 97}{\ds 24} + \left(\frac{\ds 593
}{\ds 64\sqrt{3}\pi}+\frac{\ds \sqrt{3}\pi}{\ds 4}\right)\e + \Odr(\e^2), \;\; \mbox{
as } \e\to0,
\]
}
\end{examp}

\begin{examp}[{\it Bean} curve]\label{ex:bean}
{\rm As a third example we consider the quartic curve defined by
$x_{1}(x_{1}-1)(x_{1}^{2}+x_{2}^{2}) + x_{2}^{4}=0$. We now
have
\[
H(x_{1})=2h_{+}(x_{1})=2h_{-}(x_{1})=(2x_{1})^{1/2}\left[ 1-x_{1}+\left(
1-x_{1}\right)^{1/2}\left(1+3x_{1}\right)^{1/2}\right]^{1/2}.
\]
The maximum of $H$ is situated at $x_{1}=2/3$ and we obtain
\[
\lambda_{1}(\e) = \frac{\ds 9\pi^{2}}{\ds
16\e^2}+\frac{\ds 3\sqrt{15}\pi}{\ds 8\e} +\frac{\ds 127}{\ds 40} + \left(\frac{\ds 24229
}{\ds 600\sqrt{15}\pi}+\frac{\ds 5\sqrt{5}\pi}{\ds 16\sqrt{3}}\right)\e + \Odr(\e^2), \;\; \mbox{
as } \e\to0,
\]
}
\end{examp}

\begin{examp}[Convex, with $h_{+}\neq h_{-}$]
{\rm Let $h_{+}(x_{1}) = \sin(\pi x_{1})$ and $h_{-}(x_{1}) = \pi(1-x_{1})/2$,
yielding $H(x_{1})= \sin(\pi x_{1}) + \pi(1-x_{1})/2$, which has a maximum at
$x_{1}=1/3$. The expression for the eigenvalue asymptotics now becomes
\[
\begin{array}{lll}
\lambda_{1}(\e) & = & \frac{\ds 36\pi^{2}}{\ds (3\sqrt{3}+2\pi)^{2}\e^{2}} +
\frac{\ds 6 3^{3/4}\pi^{2}}{\ds (3\sqrt{3}+2\pi)^{3/2}\e} +
\pi^{2}\left[\frac{\ds 9(27+6\sqrt{3}\pi+16\pi^{2})}{\ds 16(3\sqrt{3}+2\pi)^{2}}-\frac{\ds 19}{\ds 216}\right]\eqskip
& & \hspace*{0.5cm} + \left(\frac{\ds 13273}{\ds 4608 } + \frac{\ds 1807\pi }{\ds 3^{1/2}2304 }+
\frac{\ds 5465\pi^{2} }{\ds 1296 }+\frac{\ds 17257\pi^{3} }{\ds 3^{1/2}11664 }
\right)\frac{\ds 3^{-1/4}\pi^{2}\e}{\ds (3\sqrt{3}+2\pi)^{3/2}}\eqskip
& & \hspace*{1cm} + \Odr(\e^2), \;\; \mbox{
as } \e\to0.
\end{array}
\]
}
\end{examp}

\begin{examp}[Non convex]
{\rm Let $h_{+}(x_{1}) = 1+\sin(7\pi x_{1}/2)$ and $h_{-}(x_{1}) = 7\pi(1-x_{1})/4$,
yielding $H(x_{1})= 1+ \sin(7\pi x_{1}/2) + 7\pi(1-x_{1})/4$, which has its global
maximum on $(0,1)$ at $x_{1}=2/21$. We have
\[
\lambda_{1}(\e) = \frac{\ds 0.210941}{\ds \e^{2}}+\frac{\ds 1.79692}{\ds \e}+4.35119 + 60.5706\e
+ \Odr(\e^2), \;\; \mbox{ as } \e\to0,
\]
where, for simplicity, we only presented the numerical values of the coefficients.
}
\end{examp}

\end{document}